\documentclass[12pt]{article}

\usepackage[latin1]{inputenc}
\usepackage{amsfonts}
\usepackage{amsmath}
\usepackage{amssymb}
\usepackage{amsthm}
\usepackage{mathrsfs}
\usepackage[american]{babel}
\usepackage[T1]{fontenc}
\usepackage{graphicx}

\newtheorem{theorem}{Theorem}[section]

\newtheorem{conj}[theorem]{Conjecture}
\newtheorem{definition}[theorem]{Definition}
\newtheorem{lemma}[theorem]{Lemma}
\newtheorem{claim}[theorem]{Claim}
\newtheorem{prop}[theorem]{Proposition}
\newtheorem{cor}[theorem]{Corollary}
\newtheorem{prp}[theorem]{Property}
\author{\ \\ \\
Glenn Hurlbert\thanks{\texttt{hurlbert@asu.edu} (Corresponding author)}
\ and
Vikram Kamat\thanks{\texttt{vikram.kamat@asu.edu}}\\
{\small School of Mathematical and Statistical Sciences}\\
{\small Arizona State University, Tempe, Arizona 85287-1804}\\ \\ \\
}

\title{Erd\"{o}s-Ko-Rado theorems\\for chordal and bipartite graphs}

\begin{document}

\maketitle

\newpage

\begin{abstract}
One of the more recent generalizations of the Erd\"{o}s-Ko-Rado
theorem, formulated by Holroyd, Spencer and Talbot \cite{hst}, defines the Erd\"os-Ko-Rado property for
graphs in the following manner: for a graph $G$, vertex $v\in G$ and some integer $r\geq 1$, denote
the family of independent $r$-sets of $V(G)$ by
$\mathcal{J}^{(r)}(G)$ and the subfamily $\{A\in
\mathcal{J}^{(r)}(G):v\in A\}$ by $\mathcal{J}_v^{(r)}(G)$, called a star.
Then, $G$ is said to be $r$-EKR if no intersecting subfamily
of $\mathcal{J}^{(r)}(G)$ is larger than the largest star in
$\mathcal{J}^{(r)}(G)$.
In this paper, we prove that if $G$ is a disjoint union of chordal graphs, including at least one singleton,
then $G$ is $r$-EKR if $r\leq \frac{\mu(G)}{2}$, where $\mu(G)$ is the minimum size of a maximal independent set.

We will also prove Erd\"{o}s-Ko-Rado results for chains of complete graphs,
which are a class of chordal graphs obtained by blowing up edges of a path into complete graphs.
We also consider similar problems for ladder graphs and trees, and prove preliminary results for these graphs.

\end{abstract}

\pagebreak

\section{Introduction} Let $X=[n]=\{1,\ldots,n\}$ be a
set of size $n$. We denote the power set of $X$ by
$\mathcal{P}=\mathcal{P}(X)=\{A|A\subseteq X\}$. A family
$\mathcal{A}$ is a collection of sets in $\mathcal{P}$.
$\mathcal{A}$ is said to be an intersecting family if $A,B\in
\mathcal{A}$ imply $A\cap B \neq \emptyset$. An intersecting
$r$-uniform hypergraph is an intersecting family where all sets have
cardinality $r$. The problem of finding how large an intersecting
family can be is trivial: an intersecting family can have size at
most $2^{n-1}$ with $\mathcal{P}(X_x)=\{A: A\subset X,x\in A\}$
being one of the extremal families.

If we consider this problem for intersecting $r$-uniform
hypergraphs, we see that the problem is trivial for $n\leq 2r$
because the set of all $r$-sets in $X$, denoted by $X^{(r)}$, is
intersecting for $n<2r$, and if $n=2r$, every family
contains exactly one of any two complimentary sets, so the maximum size is
at most
$\frac{1}{2}{n \choose r}={n-1 \choose r-1}.$

If $n>2r$, then the problem is solved by the Erd\"{o}s-Ko-Rado
Theorem \cite{ekr}, one of the seminal results in extremal set
theory.

\begin{theorem}{(Erd\"{o}s-Ko-Rado theorem \cite{ekr})}\label{ekrt}
Let $2\leq r< n/2$ and let $\mathcal{A}\subset X^{(r)}$ be
an intersecting hypergraph. Then
$$|\mathcal{A}|\leq {n-1 \choose r-1}$$ with equality iff
$\mathcal{A}=X_x^{(r)}=\{A|A\in X^{(r)}, x\in A\}$ for some $x\in
X$.
\end{theorem}

There have been generalizations of the theorem in different
directions. Deza and Frankl \cite{df} give
a very nice survey of the EKR-type results proved in the 1960s, 70's and
80's. In this paper, we concern ourselves with the generalization for graphs, formulated by Holroyd, Spencer and Talbot in \cite{hst}.
\subsection{Erd\"{o}s-Ko-Rado property for graphs}
The Erd\"{o}s-Ko-Rado property for graphs is defined in the following manner.

For a graph $G$, vertex $v\in V(G)$ and some integer $r\geq 1$,
denote the family of independent $r$-sets of $V(G)$ by
$\mathcal{J}^{(r)}(G)$ and the subfamily $\{A\in
\mathcal{J}^{(r)}(G):v\in A\}$ by $\mathcal{J}_v^{(r)}(G)$, called a
star. Then, $G$ is said to be $r$-EKR if no intersecting subfamily
of $\mathcal{J}^{(r)}(G)$ is larger than the largest star in
$\mathcal{J}^{(r)}(G)$. If every maximum sized intersecting
subfamily of $\mathcal{J}^{(r)}(G)$ is a star, then $G$ is said to
be strictly $r$-EKR. This can be viewed as the Erd\"{o}s-Ko-Rado
property on a ground set, but with additional structure on this
ground set. In fact, the Erd\"{o}s-Ko-Rado theorem can be restated
in these terms as follows.

\begin{theorem}{(Erd\"{o}s-Ko-Rado theorem \cite{ekr})}
The graph on $n$ vertices with no edges is $r$-EKR if $n\geq 2r$ and
strictly $r$-EKR if $n>2r.$
\end{theorem}

There are some results giving EKR-type theorems for different types
of graphs. The following theorem was originally proved by Berge
\cite{berge}, with Livingston \cite{living} characterizing the
extremal case.
\begin{theorem}{(Berge \cite{berge},Livingston \cite{living})}
\label{berth} If $r\geq 1$, $t\geq 2$ and $G$ is the disjoint union
of $r$ copies of $K_t$, then $G$ is $r$-EKR and strictly so unless
$t=2$.
\end{theorem}

Other proofs of this result were given by Gronau \cite{gron} and
Moon \cite{moon}. Berge \cite{berge} proved a stronger result.

\begin{theorem}{(Berge \cite{berge})}\label{berth2}
If $G$ is the disjoint union of $r$ complete graphs each of
order at least $2$, then $G$ is $r$-EKR.
\end{theorem}

A generalization of Theorem \ref{berth} was first stated by Meyer
\cite{meyer} and proved by Deza and Frankl \cite{df}.
\begin{theorem}{(Meyer \cite{meyer},Deza and Frankl \cite{df})}
\label{dfth} If $r\geq 1$, $t\geq 2$ and $G$ is a disjoint union of
$n\geq r$ copies of $K_t$, then $G$ is $r$-EKR and strictly so
unless $t=2$ and $r=n$.
\end{theorem}

In the paper which introduced the notion of the $r$-EKR property for
graphs, Holroyd, Spencer and Talbot \cite{hst} prove a generalization of Theorems
\ref{berth2} and \ref{dfth}.

\begin{theorem}{(Holroyd et al. \cite{hst})} If $G$ is a disjoint union of
$n\geq r$ complete graphs
each of order at least $2$, then $G$ is $r$-EKR.
\end{theorem}

The compression technique used in \cite{hst}, which is equivalent to
contracting an edge in a graph, was employed by Talbot\cite{tal} to
prove a theorem for the $k^{th}$ power of a cycle.

\begin{definition}
The $k^{th}$ power of a cycle $C_n^k$ is a graph with vertex set
$[n]$ and edges between $a,b\in [n]$ iff $1\leq |a-b \textrm{ mod }
n|\leq k$.
\end{definition}
\begin{theorem}{(Talbot \cite{tal})}
If $r,k,n\geq 1$, then $C_n^k$ is $r$-EKR and strictly so unless
$n=2r+2$ and $k=1$.
\end{theorem}

An analogous theorem for the $k^{th}$ power of a path is also proved
in \cite{hst}.
\begin{definition}
The $k^{th}$ power of a path $P_n^k$ is a graph with vertex set
$[n]$ and edges between $a,b\in [n]$ iff $1\leq |a-b|\leq k$.
\end{definition}
\begin{theorem}{(Holroyd et al. \cite{hst})}
If $r,k,n\geq 1$, then $P_n^k$ is $r$-EKR.
\end{theorem}

It can be observed here that the condition $r\leq n/2$ is not
required for the graphs $C_n^k$ and $P_n^k$ because for each of the two graphs,
there is no independent set of size greater than $n/2$, so the $r$-EKR property holds vacuously if $r>n/2$.

The compression proof technique is also employed to prove a result
for a larger class of graphs.

\begin{theorem}{(Holroyd et al. \cite{hst})}\label{mix}
If $G$ is a disjoint union of $n\geq 2r$ complete graphs,
cycles and paths, including an isolated singleton, then $G$ is
$r$-EKR.
\end{theorem}

The problem of finding if a graph $G$ is $2$-EKR is addressed by Holroyd and Talbot in
\cite{ht}.

\begin{theorem}{(Holroyd and Talbot \cite{ht})} \label{2ekr}
Let $G$ be a non-complete graph of order $n$ with minimum degree
$\delta$ and independence number $\alpha$.
\begin{enumerate}
\item If $\alpha=2$, then $G$ is strictly $2$-EKR.
\item If $\alpha\geq 3$, then $G$ is $2$-EKR if and only if
$\delta\leq n-4$ and strictly so if and only if $\delta\leq n-5$,
the star centers being the vertices of minimum degree.
\end{enumerate}
\end{theorem}

Holroyd and Talbot also present an interesting conjecture in \cite{ht}.
\begin{definition}
The minimum size of a maximal independent vertex set of a graph $G$
is the minimax independent number, denoted by $\mu(G)$.
\end{definition}
It can be noted here that $\mu(G)=i(G)$, where $i(G)$ is the independent domination number.
\begin{conj}\label{minmax}
Let $G$ be any graph and let $1\leq r\leq \frac{1}{2}\mu$; then $G$
is $r$-EKR(and is strictly so if $2<r<\frac{1}{2}\mu$).
\end{conj}

This conjecture seems hard to prove or disprove; however, restricting attention to
certain classes of graphs makes the problem easier to tackle. Borg and Holroyd \cite{bh}
prove the conjecture for a large class of graphs, which contain a singleton as a component.

\begin{definition}(Borg, Holroyd \cite{bh})\label{mnd}
For a monotonic non-decreasing (\emph{\textsf{mnd}}) sequence $\mathbf{d} = \{d_i\}_{i\in \mathbb{N}}$
of non-negative integers, let $M = M(\mathbf{d})$ be the graph such that $V(M)=\{x_i:i\in \mathbb{N}\}$
and for $x_a,x_b\in V(M)$ with $a<b$, $x_ax_b\in E(M)$ iff $b\leq a+d_a$. Let $M_n=M_n(\mathbf{d})$
be the subgraph of $M$ induced by the subset $\{x_i:i\in [n]\}$ of $V(M)$.
Call $M_n$ an \emph{\textsf{mnd}} graph.
\end{definition}
\begin{definition}{(Borg, Holroyd \cite{bh})}
For $n>2$, $1\leq k<n-1$, $0\leq q<n$, let $C_{q,n}^{k,k+1}$ be the graph with vertex set
$\{v_i:i\in [n]\}$ and edge set $E(C_n^k)\cup \{v_{i}v_{i+k+1 \textrm{ mod } n}:1\leq i\leq q\}$.
If $q>0$, call $C_{q,n}^{k,k+1}$ a modified $k^{th}$ power of a cycle.
\end{definition}
Borg and Holroyd \cite{bh} prove the following theorem.

\begin{theorem}\label{bhth}
Conjecture \ref{minmax} is true if $G$ is a disjoint union of complete multipartite graphs,
copies of \emph{\textsf{mnd}} graphs, powers of cycles, modified powers of cycles, trees,
and at least one singleton.
\end{theorem}
One of our main results in this paper extends the class of graphs which satisfy
Conjecture \ref{minmax} by proving the conjecture for all chordal graphs which contain a singleton.
It can be noted that the \textsf{mnd} graphs in Theorem \ref{bhth} are chordal.

We also define a special class of chordal graphs, and prove a stronger EKR result for these graphs.
Finally, we consider similar problems for two classes of bipartite graphs, trees and ladder graphs.
\subsection{Main Results}
\begin{definition}
A graph $G$ is a chordal graph if every cycle of length at least $4$ has a chord.
\end{definition}

It is easy to observe that if $G$ is chordal, then every induced subgraph of $G$ is also chordal.
\begin{definition}
A vertex $v$ is called \textit{simplicial} in a graph $G$ if its neighborhood is a clique in $G$.
\end{definition}
Consider a graph $G$ on $n$ vertices, and let $\sigma=[v_1,\ldots,v_n]$ be an ordering
of the vertices of $G$.
Let the graph $G_i$ be the subgraph obtained by removing the vertex set $\{v_1,\ldots,v_{i-1}\}$ from $G$.
Then $\sigma$ is called a \textit{simplicial elimination ordering}
if $v_i$ is simplicial in the graph $G_i$, for each $1\leq i\leq n$.
We state a well known characterization for chordal graphs, due to Dirac \cite{dirac}.
\begin{theorem}\label{dirc}
A graph $G$ is a chordal graph if and only if it has a simplicial elimination ordering.
\end{theorem}
It is easy to see, using this characterization of chordal graphs,
that the \textsf{mnd} graphs of Definition \ref{mnd} are chordal.
\begin{prop}
If $M_n$ is an \emph{\textsf{mnd}} graph on $n$ vertices, $M_n$ is chordal.
\end{prop}
\begin{proof}
It can be seen that ordering the vertices of $M_n$, according to
the corresponding degree sequence $\mathbf{d}$, as stated in Definition \ref{mnd},
gives a simplicial elimination ordering.
\end{proof}

Note that, with or without the non-decreasing condition on the sequence $\mathbf{d}$, the resulting graph
is an interval graph --- use the interval $[a,a+d_a]$ for vertex $x_a$ --- which is chordal regardless.

We prove the non-strict part of Conjecture \ref{minmax} for disjoint unions of chordal graphs,
containing at least one singleton.
\begin{theorem}\label{mainthm}
If $G$ is a disjoint union of chordal graphs, including at least one singleton,
and if $r\leq \frac{1}{2}\mu(G)$, then $G$ is $r$-EKR.
\end{theorem}

We also consider graphs which do not have singletons.
Consider a class of chordal graphs constructed as follows.

Let $P_{n+1}$ be a path on $n$ edges with
$V(P_{n+1})=\{v_1,\ldots,v_{n+1}\}$. Label the edge $v_iv_{i+1}$ as
$i$, for each $1\leq i\leq n$. A \textit{chain}
of complete graphs, of length $n$, is obtained from $P_{n+1}$ by replacing each edge of
$P_{n+1}$ by a complete graph of order at least $2$ in the following
manner: to convert edge $i$ of $P_{n+1}$ into $K_s$, introduce a
complete graph $K_{s-2}$ and connect $v_i$ and $v_{i+1}$
to each of the $s-2$ vertices of the complete graph. Call the resulting complete
graph $G_i$, and call each $G_i$ a link of the chain. We call
$v_i$ and $v_{i+1}$
the \textit{connecting vertices}
of this complete graph, with the exception of $G_1$ and $G_n$, which
have only one connecting vertex each (the ones shared with $G_2$ and
$G_{n-1}$ respectively). In general, for each $2\leq i\leq n$, call $v_i$ the $(i-1)^{th}$ connecting vertex of $G$.
Unless otherwise specified, we will refer to a chain of complete graphs as just a chain.
We will call an isolated vertex a \textit{trivial chain} (of length $0$), while a complete graph
is simply a chain of length $1$. Call a chain of length $n$ \textit{special}
if $n\in \{0,1\}$ or if $n\geq 2$ and the following conditions hold:

\begin{enumerate}
\item $|G_{i}|\geq |G_{i-1}|+1$ for each $2\leq i\leq n-1$;
\item $|G_n|\geq |G_{n-1}|$.
\end{enumerate}
We prove the following results for special chains.
\begin{theorem} \label{mthm}
If $G$ is a \textit{special} chain, then
$G$ is $r$-EKR for all $r\geq 1$.
\end{theorem}

\begin{theorem}\label{mthm2}
If $G$ is a disjoint union of $2$ \textit{special} chains, then $G$ is $r$-EKR for all $r\geq 1$.
\end{theorem}

We will also consider similar problems for bipartite graphs.
A basic observation about complete bipartite graphs, and its obvious
generalization for complete multipartite graphs, are mentioned below.
\begin{itemize}
\item If $G=K_{m,n}$ and $m\leq n$, then $G$ is $r$-EKR for all
$r\leq \frac{m}{2}$.
\item If $G=K_{m_1,\ldots,m_k}$, with $m_1\leq m_2\leq \ldots \leq
m_k$, then $G$ is $r$-EKR for all $r\leq \frac{m_1}{2}$.
\end{itemize}

It is easy to see why these hold. If $\mathscr{B}\subseteq
\mathscr{J}^r(G)$ is intersecting, then each $A\in \mathscr{B}$ lies
in the same partite set. Clearly, if $2r\leq m\leq n$, then $G$ is
$r$-EKR by Theorem \ref{ekrt}. A similar argument works for complete
multipartite graphs as well.

Holroyd and Talbot \cite{ht} proved Conjecture \ref{minmax} for a disjoint union of two complete multipartite graphs.

If we consider non-complete bipartite graphs with high minimum degree, it seems that they usually
have low $\mu$ (always at most $\min\{n-\delta,n/2\}$).
Instead, in this paper, we consider bipartite graphs with low maximum
degree in order to have higher values of $\mu$
(always at least $\frac{n}{\Delta+1}$).
In particular, we look at trees and ladder graphs, two such classes of sparse bipartite graphs.

One of the difficult problems in dealing with graphs without singletons is that of finding centers of maximum stars.
We consider this problem for trees, and conjecture that there is a maximum star in a tree that is centered at a leaf.

\begin{conj}\label{stree}
For any tree $T$ on $n$ vertices, there exists a leaf $x$ such that for any $v\in V(T)$,
$|\mathcal{J}^r_v(T)|\leq |\mathcal{J}^r_v(T)|$.
\end{conj}

We prove this conjecture for $r\leq 4$.
\begin{theorem} \label{tstar} Let $1\leq r\leq 4$. Then, a maximum sized
star of $r$-independent vertex sets of $T$ is centered at a leaf.
\end{theorem}

We will also prove that the ladder graph is $3$-EKR.
\begin{definition}
The ladder graph $L_n$ with $n$ rungs can be defined as the
cartesian product of $K_2$ and $P_n$.
\end{definition}
It is not hard to see that, for $L_n$, $\mu(L_n)\leq \lceil{\frac{n+1}{2}\rceil}$.
In fact, we show
that equality holds.
\begin{prop}
$$\mu(L_n)=\left\lceil\frac{n+1}{2}\right\rceil.
$$
\end{prop}
\begin{proof}
The result is trivial if $n\leq 2$, so let $n\geq 3$.
Suppose $\mu(L_n)<\lceil{\frac{n+1}{2}\rceil}$ and let $A$ be a maximal independent set of size $\mu(L_n)$.
Then, there exist two consecutive rungs, say the $i^{th}$ and $(i+1)^{st}$ in $L_n$,
with endpoints $\{x_i,y_i\}$ and $\{x_{i+1},y_{i+1}\}$ respectively, such that
$\{x_i,y_i\}\cap A=\emptyset$ and $\{x_{i+1},y_{i+1}\}\cap A=\emptyset$.
Let $u=x_i$, $v=x_{i-1}$ and $w=y_i$ if $i>1$, otherwise, let $u=x_{i+1}$, $v=x_{i+2}$ and $w=y_{i+1}$.
$A\cup \{u\}$ is not independent, since $A$ is maximal.
Then, $v\in A$ and $A\cup \{w\}$ is independent, a contradiction.
\end{proof}
\begin{theorem}\label{ladder}
The graph $L_n$ is $3$-EKR for all $n\geq 1$.
\end{theorem}

The rest of the paper is organized as follows: in Section \ref{s3},
we give a proof of Theorem \ref{mainthm}, in Section \ref{s4},
we give proofs of Theorems \ref{mthm} and \ref{mthm2}, and in Section \ref{s5},
we give proofs of Theorems \ref{ladder} and \ref{tstar}.
\section{Proof of Theorem \ref{mainthm}}\label{s3}
We begin by fixing some notation.
For a graph $G$ and a vertex $v\in V(G)$, let $G-v$ be the graph obtained from $G$ by removing vertex $v$.
Also, let $G\downarrow v$ denote the graph obtained by removing $v$ and its set of neighbors from $G$.
We note that if $G$ is a disjoint union of chordal graphs and if $v\in G$, the graphs $G-v$ and
$G\downarrow v$ are also disjoint unions of chordal graphs.

We state and prove a series of lemmas, which we will use in the proof of Theorem \ref{mainthm}.
\begin{lemma}\label{l1}
Let $G$ be a graph containing an isolated vertex $x$.
Then, for any vertex $v\in V(G)$, $|\mathcal{J}^r_v(G)|\leq |\mathcal{J}^r_x(G)|$.
\end{lemma}
\begin{proof}
Let $v\in V(G)$, $v\neq x$.
We define a function $f:\mathcal{J}^r_v(G)\to \mathcal{J}^r_x(G)$ as follows.
\begin{displaymath}
   f(A) = \left \{
     \begin{array}{ll}
       A & \textrm{ if } x\in A \\
       A\setminus \{v\}\cup \{x\} & \textrm{ otherwise}
     \end{array}
   \right.
\end{displaymath}
It is easy to see that the function is injective, and this completes the proof.
\end{proof}
\begin{lemma}\label{l2}
Let $G$ be a graph, and let $v_1,v_2\in G$ be vertices such that $N[v_1]\subseteq N[v_2]$.
Then, the following inequalities hold:
\begin{enumerate}
\item $\mu(G-v_2)\geq \mu(G)$;
\item $\mu(G\downarrow v_2)+1\geq \mu(G)$.
\end{enumerate}
\end{lemma}
\begin{proof}
We begin by noting that the condition $N[v_1]\subseteq N[v_2]$ implies that $v_1v_2\in E(G)$.
\begin{enumerate}
\item We will show that if $I$ is a maximal independent set in $G-v_2$, then $I$ is maximally independent in $G$.
Suppose $I$ is not a maximal independent set in $G$.
Then, $I\cup \{v_2\}$ is an independent set in $G$.
Thus, for any $u\in N[v_2]$, $u\notin I$.
In particular, for any $u\in N[v_1]$, $u\notin I$.
Thus, $I\cup \{v_1\}$ is an independent set in $G-v_2$.
This is a contradiction. Thus, $I$ is a maximal independent set in $G$.

Taking $I$ to be the smallest maximal independent set in $G-v_2$, we get $\mu(G-v_2)=|I|\geq \mu(G)$.
\item We will show that if $I$ is a maximal independent set in $G\downarrow v_2$,
then $I\cup \{v_2\}$ is a maximal independent set in $G$.
Of course,
$I\cup \{v_2\}$ is independent, so suppose it is not maximal.
Then, for some vertex $u\in G\downarrow v_2$ and
$u\notin I\cup \{v_2\}$, $I\cup \{u,v_2\}$ is an independent set.
Thus, $I\cup \{u\}$ is an independent set in $G\downarrow v_2$, a contradiction.

Taking $I$ to be the smallest maximal independent set in $G\downarrow v_2$,
we get $\mu(G\downarrow v_2)+1=|I|+1\geq \mu(G)$.
\end{enumerate}
\end{proof}
\begin{cor}\label{cor1}
Let $G$ be a graph, and let $v_1,v_2\in G$ be vertices such that $N[v_1]\subseteq N[v_2]$.
Then, the following statements hold:
\begin{enumerate}
\item If $r\leq \frac{1}{2}\mu(G)$, then $r\leq \frac{1}{2}\mu(G-v_2)$;
\item If $r\leq \frac{1}{2}\mu(G)$, then $r-1\leq \frac{1}{2}\mu(G\downarrow v_2)$.
\end{enumerate}
\end{cor}
\begin{proof}
\begin{enumerate}
\item This follows trivially from the first part of Lemma \ref{l2}.
\item To prove this part, we use the second part of Lemma \ref{l2} to show
$$r-1\leq \frac{1}{2}\mu(G)-1=\frac{\mu(G)-2}{2}\leq \frac{\mu(G\downarrow v_2)}{2}-\frac{1}{2}.$$
\end{enumerate}
\end{proof}
Let $H$ be a component of $G$, so $H$ is
a chordal graph on $m$ vertices, $m\geq 2$.
Let $\{v_1,\ldots,v_m\}$ be a simplicial elimination ordering of $H$ and let $v_1v_i\in E(H)$ for some $i\geq 2$.
Let $\mathcal{A}\subseteq \mathcal{J}^r(G)$ be an intersecting family.
We define a compression operation $f_{1,i}$ for the family $\mathcal{A}$.
Before we give the definition, we note that if $A$ is an independent set and if $v_i\in A$,
then $A\setminus \{v_i\}\cup \{v_1\}$ is also independent.
\begin{displaymath}
   f_{1,i}(A) = \left \{
     \begin{array}{ll}
       A\setminus \{v_i\}\cup \{v_1\} & \textrm{ if } v_i\in A, v_1\notin A, A\setminus \{v_i\}\cup \{v_1\}\notin \mathcal{A}  \\
       A & \textrm{ otherwise}
     \end{array}
   \right.
\end{displaymath}
Then, we define the family $\mathcal{A}'$ by
$$\mathcal{A}'=f_{1,i}(\mathcal{A})=\{f_{1,i}(A):A\in \mathcal{A}\}.$$

It is not hard to see that $|\mathcal{A}'|=|\mathcal{A}|$. Next, we define the families

$$\mathcal{A}'_{i}=\{A\in \mathcal{A}':v_i\in A\},$$
$$\bar{\mathcal{A}'_{i}}=\mathcal{A}'\setminus \mathcal{A}'_{i},\ {\rm and}$$
$$\mathcal{B}'=\{A\setminus \{v_i\}:A\in \mathcal{A}'_{i}\}.$$

Then we have
\begin{eqnarray}
|\mathcal{A}| &=& |\mathcal{A}'| \nonumber \\
&=& |\mathcal{A}'_{i}|+|\bar{\mathcal{A}'_{i}}| \nonumber \\
&=& |\mathcal{B}'|+|\bar{\mathcal{A}'_{i}}|. \label{eq1}
\end{eqnarray}
 We prove the following lemma about these families.
\begin{lemma}\label{l3}
\begin{enumerate}
\item $\bar{\mathcal{A}'_{i}}\subseteq \mathcal{J}^r(G-v_i)$.
\item $\mathcal{B}'\subseteq \mathcal{J}^{(r-1)}(G\downarrow v_i)$.
\item $\bar{\mathcal{A}'_{i}}$ is intersecting.
\item $\mathcal{B}'$ is intersecting.
\end{enumerate}
\end{lemma}
\begin{proof}
It follows
from the definitions of the families that $\bar{\mathcal{A}'_{i}}\subseteq \mathcal{J}^r(G-{v_i})$
and $\mathcal{B}'\subseteq \mathcal{J}^{(r-1)}(G\downarrow v_i)$.
So, we only prove that the two families are intersecting.
Consider $A,B\in \bar{\mathcal{A}'_{i}}$.
If $v_1\in A$ and $v_1\in B$, we are done.
If $v_1\notin A$ and $v_1\notin B$, then $A,B\in \mathcal{A}$ and hence $A\cap B\neq \emptyset$.
So, suppose $v_1\notin A$ and $v_1\in B$.
Then, $A\in \mathcal{A}$.
Also, either $B\in \mathcal{A}$, in which case we are done or $B_1=B\setminus \{v_1\}\cup \{v_i\}\in \mathcal{A}$.
Then, $|A\cap B|=|A\cap B\setminus \{v_1\}\cup \{v_i\}|=|A\cap B_1|>0$.

Finally, consider $A,B\in \mathcal{B}'$.
Since $A\cup \{v_i\}\in \mathcal{A}'_{v_i}$, $A\cup \{v_1\}\in \mathcal{A}$ and $A\cup \{v_i\}\in \mathcal{A}$.
A similar argument works for $B$.
Thus, $|(A\cup \{v_1\})\cap (B\cup \{v_i\})|>0$ and hence, $|A\cap B|>0$.
\end{proof}
The final lemma we prove is regarding the star family $\mathcal{J}^r_x(G)$, where $x$ is an isolated vertex.
\begin{lemma}\label{l4}
Let $G$ be a graph containing an isolated vertex $x$ and let $v\in V(G)$, $v\neq x$. Then, we have
$$|\mathcal{J}^r_x(G)|=|\mathcal{J}^r_x(G-v)|+|\mathcal{J}^{(r-1)}_x(G\downarrow v)|.$$
\end{lemma}
\begin{proof}
Partition the family $\mathcal{J}^r_x(G)$ into two parts.
Let the first part contain all sets containing $v$, say $\mathcal{F}_v$,
and let the second part contain all sets which do not contain $v$, say $\bar{\mathcal{F}_v}$.
Then

$\mathcal{F}_v=\mathcal{J}^{(r-1)}_x(G\downarrow v)$ and $\bar{\mathcal{F}_v}=\mathcal{J}^r_x(G-v)$.
\end{proof}
We proceed to a proof of Theorem \ref{mainthm}.
\subsection*{Proof of Theorem \ref{mainthm}}
\begin{proof}
The theorem trivially
holds for $r=1$, so suppose $r\geq 2$.
Let $G$ be a disjoint union of chordal graphs, including at least one singleton, and let $\mu(G)\geq 2r$.
We do induction on $|G|$.
If $|G|=\mu(G)$, then $G=E_{|G|}$, and we are done by the Erd\"{o}s-Ko-Rado theorem.
So, suppose $|G|>\mu(G)$, and there is one component, say $H$, which is a chordal graph having $m$ vertices, $m\geq 2$.
Let $\{v_1,\ldots,v_m\}$ be a simplicial ordering of $H$ and suppose $v_1v_i\in E(H)$ for some $i\geq 2$.
Since
the neighborhood of $v_1$ is a clique, we have $N[v_1]\subseteq N[v_i]$.
Also, let $x$  be an isolated vertex in $G$.
Let $\mathcal{A}\subseteq \mathcal{J}^r(G)$ be intersecting.

Define the compression operation $f_{1,i}$ and the families $\bar{\mathcal{A}'_{i}}$
and $\mathcal{B}'$ as before.
Using Equation \ref{eq1}, Lemmas \ref{l1}, \ref{l2}, \ref{l3}, \ref{l4},
Corollary \ref{cor1} and the induction hypothesis, we have
\begin{eqnarray}
|\mathcal{A}|&=&|\bar{\mathcal{A}'_{i}}|+|\mathcal{B}'| \nonumber \\
&\leq & |\mathcal{J}^r_x(G-v_i)|+|\mathcal{J}^{(r-1)}_x(G\downarrow v_i)|  \nonumber \\
&=& |\mathcal{J}^r_x(G)|. \label{eq2}
\end{eqnarray}
\end{proof}
\section{Proofs of Theorems \ref{mthm} and \ref{mthm2}}\label{s4}
The main technique we use to prove Theorem \ref{mthm} is a compression operation
that is equivalent to compressing a clique to a single vertex.
In a sense, it is a more general version of the technique used in \cite{hst}.
We begin by stating and proving a technical lemma, similar to the one proved in \cite{hst}.
We will then use it to prove Theorem \ref{mthm} by induction.
\subsection{A technical lemma}

Let $H\subseteq G$ with $V(H)=\{v_1,\ldots,v_s\}$. Let $G/H$ be the
graph obtained by contracting the subgraph $H$ to a single vertex.
The contraction function $c$ is defined as follows.
\begin{displaymath}
   c(x) = \left \{
     \begin{array}{ll}
       v_1 & : x \in H\\
       x & : x \notin H
     \end{array}
   \right.
\end{displaymath}
When we contract $H$ to $v_1$, the edges which have both endpoints in $H$ are lost
and if there is an edge $xv_i\in E(G)$ such that $x\in V(G)\setminus V(H)$,
then there is an edge $xv_1\in E(G/H)$. Duplicate edges are disregarded.

Also, let $G-H$ be the (possibly disconnected) graph obtained from
$G$ by removing all vertices in $H$.

\begin{lemma}\label{tl1}
Let $G=(V,E)$ be a graph and let $\mathcal{A}\subseteq
\mathcal{J}^r(G)$ be an intersecting family of maximum size. If
$H$ is a subgraph of $G$ with vertex set $\{v_1,\ldots,v_s\}$, and if $H$ is isomorphic to $K_s$,
then there exist families $\mathcal{B}$,
$\{\mathcal{C}_i\}_{i=2}^s$, $\{\mathcal{D}_i\}_{i=2}^s$,
$\{\mathcal{E}_i\}_{i=2}^s$ satisfying:
\begin{enumerate}
\item
$|\mathcal{A}|=|\mathcal{B}|+\sum_{i=2}^s|\mathcal{C}_i|+|\bigcup_{i=2}
^s\mathcal{D}_i| +\sum_{i=2}^s|\mathcal{E}_i|$;
\item $\mathcal{B}\subseteq \mathcal{J}^r(G/H)$ is intersecting; and
\item for each $2\leq i\leq s$,
\begin{enumerate}
\item $\mathcal{C}_i\subseteq \mathcal{J}^{r-1}(G-H)$ is
intersecting,
\item $\mathcal{D}_i=\{A\in \mathcal{A}:v_1\in A \textrm{ and } N(v_i)\cap
(A\setminus \{v_1\})\neq \emptyset\}$, and
\item $\mathcal{E}_i=\{A\in \mathcal{A}:v_i\in A \textrm{ and } N(v_1)\cap
(A\setminus \{v_i\})\neq \emptyset\}$.
\end{enumerate}
\end{enumerate}
\end{lemma}

To prove Lemma \ref{tl1}, we will need a claim, which we state and
prove below.

\begin{claim}\label{cl}
Let $H\subseteq G$ be isomorphic to $K_s$, $s\geq 3.$ Let
$\mathcal{A}\subseteq \mathcal{J}^r(G)$ be an intersecting family of
maximum size. Suppose $A\cup \{v_i\},A\cup \{v_j\}\in \mathcal{A}$
for some $i,j\neq 1$ and $c(A\cup \{v_i\})=A\cup \{v_1\}\in
\mathcal{J}^r(G/H).$ Then $A\cup \{v_1\}\in \mathcal{A}.$
\end{claim}
\begin{proof}
Since we have $c(A\cup \{v_i\}) \in \mathcal{J}^r(G/H)$, $B=A\cup
\{v_1\}\in \mathcal{J}^r(G).$ Suppose $B\notin \mathcal{A}$. Since
$\mathcal{A}$ is an intersecting family of maximum size,
$\mathcal{A}\cup \{B\}$ is not an intersecting family. So, there
exists a $C\in \mathcal{A}$ such that $B\cap C=\emptyset.$ So, we
have $C\cap (A\cup \{v_i\})=v_i$ and $C\cap (A\cup \{v_j\})=v_j$.
Thus, $v_i,v_j\in C.$ This is a contradiction since $v_i$ and $v_j$
are adjacent to each other.
\end{proof}

\begin{proof}{(Proof of Lemma \ref{tl1})}
Define the following families:
\begin{enumerate}
\item $\mathcal{B}=\{c(A):A\in \mathcal{A} \textrm{ and } c(A)\in
\mathcal{J}^r(G/H)\}$; and
\item for each $2\leq i\leq s$:
\begin{enumerate}
\item $\mathcal{C}_i=\{A\setminus \{v_1\}:v_1\in A \textrm{ and }
A\setminus \{v_1\}\cup \{v_i\}\in \mathcal{A}\}$,
\item $\mathcal{D}_i=\{A\in \mathcal{A}:v_1\in A \textrm{ and } N(v_i)\cap
(A\setminus \{v_1\})\neq \emptyset\}$, and
\item $\mathcal{E}_i=\{A\in \mathcal{A}:v_i\in A \textrm{ and } N(v_1)\cap
(A\setminus \{v_i\})\neq \emptyset\}$.
\end{enumerate}
\end{enumerate}

If $A,B\in \mathcal{A}$ and $A\neq B$, then $c(A)=c(B)$ iff
$A\bigtriangleup B=\{v_i,v_j\}$ for some $1\leq i,j\leq s.$ Using
this and Claim \ref{cl} (if $s\geq 3$), we have
$$|\{A\in \mathcal{A}:c(A)\in
\mathcal{J}^r(G/H)\}|=|\mathcal{B}|+\sum_{i=2}^s|\mathcal{C}_i|.$$

Also, if $A\in \mathcal{A}$, then $c(A)\notin \mathcal{J}^r(G/H)$
iff $A\in \bigcup_{i=2}^s\mathcal{D}_i \cup
\bigcup_{i=2}^s\mathcal{E}_i.$ Thus, we have
$|\mathcal{A}|=|\mathcal{B}|+\sum_{i=2}^s|\mathcal{C}_i|+|\bigcup_{i=2}
^s\mathcal{D}_i| +|\bigcup_{i=2}^s\mathcal{E}_i|$. By the definition
of the $\mathcal{E}_i$'s, $\bigcup_{i=2}^s\mathcal{E}_i$ is a
disjoint union, so we have
$$|\mathcal{A}|=|\mathcal{B}|+\sum_{i=2}^s|\mathcal{C}_i|+|\bigcup_{i=2}
^s\mathcal{D}_i| +\sum_{i=2}^s|\mathcal{E}_i|$$

It is obvious to show that $\mathcal{B}$ is intersecting since
$\mathcal{A}$ is.

Let $2\leq i\leq s.$ To see that $\mathcal{C}_i$ is intersecting,
suppose $C,D\in \mathcal{C}_i$ and $C\cap D=\emptyset$. But $C\cup
\{v_1\}$ and $D\cup \{v_i\}$ are in $\mathcal{A}$ and hence, are
intersecting. This is a contradiction.
\end{proof}

\subsection{Proof of Theorem \ref{mthm}} \label{main}

Before we move to the proof of Theorem \ref{mthm}, we will prove one
final claim regarding maximum sized star families in $G$.

\begin{claim}\label{clmstr}
If $G$ is \textit{special} chain of length $n$, then a maximum sized star is centered
at an internal vertex of $G_1$.
\end{claim}
\begin{proof}
First note that for any $i$, there is a trivial injection from a
star centered at a connecting vertex of $G_i$ to a star centered at
an internal vertex of $G_i$, which replaces the star center by that
internal vertex in every set of the family. So suppose $\mathcal{Q}$
is a star centered at a internal vertex $u$
of any of the graphs
$G_i$, $i\neq 1$. Let $G_1=K_m$. Consider the following cases.
\begin{enumerate}
\item Suppose $u$ is in $G_2$.
In this case, define an arbitrary bijection between the $m-1$ internal
vertices of $G_1$ and any $m-1$ internal vertices of $G_2$
containing $u$, such that $u$ corresponds to an internal vertex of
$G_1$, say $v$
(note that this can always be done, since if $n=2$,
then $|G_2|\geq m$, with one connecting vertex, while if $n\geq 3$,
then $|G_2|\geq m+1$, with two connecting vertices).

\item Suppose $u$ is in some $G_i$ such that
$i\geq 3$. Then, define an arbitrary bijection between the $m$ vertices
of $G_1$ and any $m$ internal vertices of $G_i$ including $u$ such
that $u$ corresponds to an internal vertex of $G_1$, say $v$.

\end{enumerate}
Next, consider any set in $\mathcal{Q}$. If it contains a vertex $w$
in $G_1$, replace that vertex by $b$ and replace $u$ by the vertex
in $G_i$ corresponding to $w$. If it does not contain a vertex in
$G_1$, replace $u$ by $v$. This defines the injection from
$\mathcal{Q}$ to a star centered at $v$.
\end{proof}

We now give a proof of Theorem \ref{mthm}.
\begin{proof}
Let $\mathcal{J}^r_1(G)$ be a maximum sized star family in $G$,
where $1$ is an internal vertex of $G_1$.

We do induction on $r$. The result is trivial for $r=1$. Let $r\geq
2$. We do induction on $n$ ($n$ is the number of links).
For $n=1$, result is vacuously true. If $n=2$, then for $r=2$, we
use Theorem \ref{2ekr} to conclude that $G$ is $2$-EKR while the
result is vacuously true for $r\geq 3$. So, let $n\geq 3$. Let
$\mathcal{A}\subseteq \mathcal{J}^r(G)$ be an intersecting family of
maximum cardinality. Let the vertices of $G_n=K_s$ be labeled from
$n_1$ to $n_s$ (let $n_1$ be the connecting vertex which also
belongs to $G_{n-1}$). Define the compression operation $c$ on $G$
and the clique $K_s$ as before. Let the families $\mathcal{B}$,
$\{\mathcal{C}_i\}_{i=2}^s$, $\{\mathcal{D}_i\}_{i=2}^s$,
$\{\mathcal{E}_i\}_{i=2}^s$ be defined as in Lemma \ref{tl1}.

Clearly, for $G$, $\mathcal{D}_i=\emptyset$ for each $2\leq i\leq
s.$ So, by Lemma \ref{tl1},
$$\mathcal{A}=\mathcal{B}+\sum_{i=2}^s|\mathcal{C}_i|+\sum_{i=2}^s|\mathcal{E}
_i|.$$

Let $G_{n-1}=K_t.$ Let the vertices of $G_{n-1}$ be labeled
from $m_1$ to $m_t$($t\leq s$), with $m_t=n_1$.
For every $1\leq i\leq t-1$ and $2\leq j\leq s$
define a set $\mathcal{H}_{ij}$ of families by
$$\mathcal{H}_{ij}=\{A\in \mathcal{A}:m_i\in A, n_j\in A \}.$$

We note that $\bigcup_{i=1}^{t-1}H_{ij}=\mathcal{E}_j$ for each $2\leq j\leq s$,
and since each of the $\mathcal{H}_{ij}$'s are also disjoint, we have
$$\sum_{i=2}^s|\mathcal{E}_i|=\sum_{1\leq i\leq t-1,2\leq j\leq
s}|\mathcal{H}_{ij}|.$$

Now, consider a complete bipartite graph $K_{t-1,s-1}$. Label the
vertices in part $1$ from $m_1$ to $m_{t-1}$ and vertices in part
$2$ from $n_2$ to $n_{s}$.

Partition the edges of the bipartite graph $K_{t-1,s-1}$ into $s-1$
matchings, each of size $t-1$.
For each matching $M_k$ ($1\leq k\leq s-1$), define the family
$$\mathcal{F}_{M_k}=\bigcup_{i,j,m_i n_j\in M_k}(\mathcal{H}_{ij}-\{n_j\}),$$
where a family $\mathcal{H}-\{a\}$
is obtained from $\mathcal{H}$ by removing $a$ from all its sets.
Then of course
$$\sum_{1\leq i\leq t-1,2\leq j\leq s}
|\mathcal{H}_{ij}|=\sum_{1\leq i\leq s-1}|\mathcal{F}_{M_i}|.$$

For each $1\leq k\leq s-1$, $F_{M_k}$ is a disjoint union and is
intersecting. The
intersecting property is obvious if both sets are in
the same $\mathcal{H}_{ij}-\{n_j\}$ since they contain $m_i$. If in
different such sets, adding distinct elements which were removed
(during the above operation) gives sets in the original family which
are intersecting.

Finally, if we consider families $C_{n_i}\cup F_{M_{i-1}}\subseteq
\mathcal{J}^{(r-1)}(G-G_n)$ for $2\leq i\leq s$, each such family is
a disjoint union. It is also intersecting since for $C\in C_{n_i}$
and $F\in F_{M_{i-1}}$, $C\cup \{n_1\}$ and $F\cup \{n_j\}$ for some
$j\neq 1$ gives us sets in $\mathcal{A}$. So, we get

\begin{eqnarray*}
|\mathcal{A}|&=&|\mathcal{B}|+\sum_{i=2}^{s}|\mathcal{C}_{n_i}|+\sum_{1\leq
i\leq s-1,2\leq j\leq s}|\mathcal{H}_{ij}| \nonumber \\
&=&|\mathcal{B}|+\sum_{i=2}^{s}|\mathcal{C}_{n_i}|+\sum_{1\leq i\leq
s-1}|\mathcal{F}_{M_i}| \nonumber \\
&=&|\mathcal{B}|+\sum_{i=2}^{s}|(C_{n_i}\cup \mathcal{F}_{M_{i-1}})|
\nonumber
\\
&\leq& \mathcal{J}_1^r(G/G_n)+(s-1)\mathcal{J}_1^{(r-1)}(G-G_n)
\nonumber \\
&=& \mathcal{J}_1^r(G).
\end{eqnarray*}

The last inequality is obtained by partitioning the star based on
whether or not it contains one of $\{n_2,\ldots,n_s\}.$
\end{proof}
\subsection{Proof of Theorem \ref{mthm2}}
\begin{proof}
We do induction on $r$.
Since the case $r=1$ is trivial, let $r\geq 2$.
Let $G$ be a disjoint union of $2$ \textit{special} chains $G'$ and $G''$,
with lengths $n_1$ and $n_2$ respectively.
We will do induction on $n=n_1+n_2$.
If $n=0$, the result holds trivially if $r=2$ and vacuously if $r\geq 3$.
So, let $n\geq 1$.
If $n=1$ or if $n_1=n_2=1$, then $\alpha(G)=2$.
In this case, $G$ is vacuously $r$-EKR  for $r\geq 3$.
Also, if $r=2$, then we are done by Theorem \ref{2ekr}.
So, without loss of generality,
we assume that $G_1$ has length at least $2$.
We can now proceed as in the proof of Theorem \ref{mthm}.
\end{proof}
\section{Bipartite graphs}\label{s5}

\subsection{Trees}
In this section, we give a proof of Theorem \ref{tstar},
which states that for a given tree $T$ and $r\leq 4$,
there is a maximum star family centered at a leaf of $T$.
\begin{proof}
The statement is trivial for $r=1$. If $r=2$, we use the fact that
for any vertex $v$, $|\mathscr{J}^{2}_v(T)|=n-1-d(v),$ where $d(v)$
is the degree of vertex $v$, and thus it will be maximum when $v$ is
a leaf.

Let $3\leq r\leq 4.$ Let $v$ be an internal vertex($d(v)\geq 2)$ and
let $\mathscr{A}=\mathscr{J}^{r}_v(T)$ be the star centered at $v$.
Consider $T$ as a tree rooted at $v$. We find an injection $f$ from
$\mathscr{A}$ to a star centered at some leaf. Let $v_1$ and $v_2$
be any two neighbors of $v$ and let $u$ be a leaf with neighbor $w$.
Let $A\in \mathscr{A}$.
\begin{enumerate}
\item If $u\in A$, then let $f(A)=A$.
\item If $u\notin A$, then we consider two cases.
\begin{enumerate}
\item If $w\notin A$, let $f(A)=A\setminus \{v\} \cup \{u\}$.
\item If $w\in A$, then $B=A\setminus \{w\} \cup \{u\}\in
\mathscr{A}$. We consider the following two cases separately.

\begin{itemize}
\item $r=3$

Let $A=\{v,w,x\}$. We know that $x$ cannot be connected to both
$v_1$ and $v_2$ since that would result in a cycle.
Without loss of generality,
suppose that
$xv_1\notin E(T)$. Then, let $f(A)=A\setminus \{v,w\} \cup
\{u,v_1\}$.

\item $r=4$

Let $A=\{v,w,w_1,w_2\}$. We first note that if there is a leaf at
distance two from $v$, then by using $1$ and $2(a)$ above, we can
show that the size of the star at this leaf is at least as much as
the given star. We again consider two cases.
\begin{itemize}
\item

Suppose that $\{v_1,v_2\}\not\subseteq N(w_1)\cup N(w_2)$.
By symmetry,
suppose $v_1\notin N(w_1)\cup N(w_2)$.
In this case, let $f(A)=A\setminus \{w,v\} \cup \{u,v_1\}$.

\item

Suppose that $\{v_1,v_2\}\subseteq N(w_1)\cup N(w_2)$.
Label so that
$v_i\in N(w_i)$ for $1\leq i\leq 2$ (in particular, $v_i$ is the
parent of $w_i$). Since neither $w_1$ nor $w_2$ is a leaf, they have
at least one child, say $x_1$ and $x_2$, respectively. In this case,
let $f(A)=\{u,x_1,x_2,v_1\}$. For this case, injection is less
obvious. We show it by contradiction as follows.
Let $f(\{v,w,w_1,w_2\})=f(\{v,w,y_1,y_2\})=\{u,x_1,x_2,v_1\}$.
We may assume that
$y_1\neq w_1$ and let $y_i$ be the child of $v_i$ and $x_i$
be the child of $y_i$; then certainly
$v_1w_1x_1y_1v_1$ gives a cycle
in $T$, a contradiction.
\end{itemize}
\end{itemize}
\end{enumerate}
\end{enumerate}
\end{proof}
\begin{figure}[ht]
\centering
\includegraphics[height=5cm]{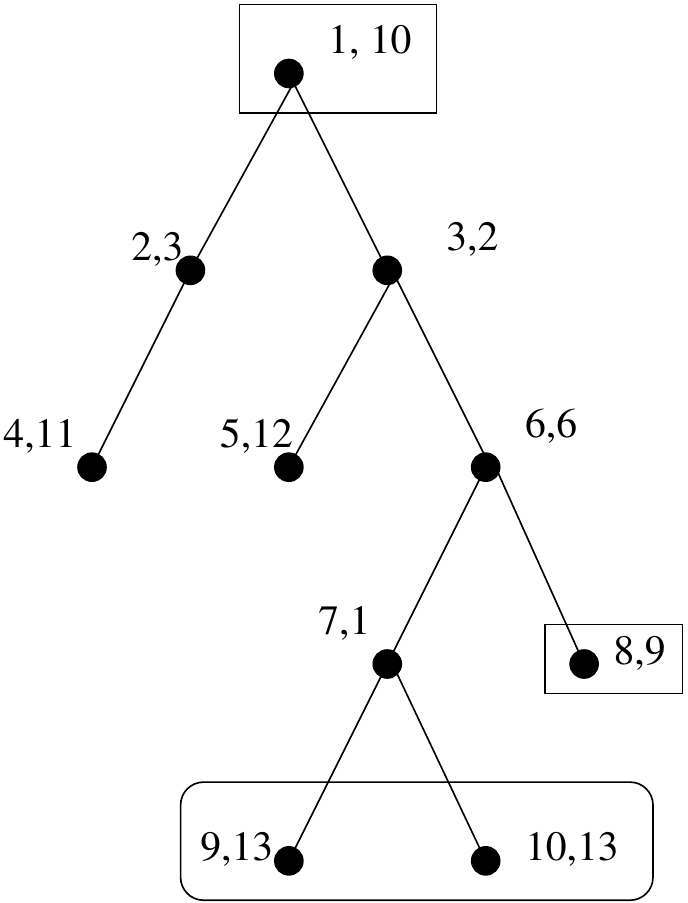}
\caption{Tree $T$ on $10$ vertices, $r=5$.}\label{figtree}
\end{figure}
We believe that Conjecture \ref{stree} holds true for all $r$.
However,
it is harder to prove because it is not true that every leaf centered star
is bigger than every non-leaf centered star; an example is illustrated in Figure \ref{figtree}.

For each vertex, the first number denotes the label,
while the second number denotes the size of the star centered at that vertex.
We note that $\mathcal{J}^5_8(T)=9$, while $\mathcal{J}^5_1(T)=10$.
However, we note that the maximum sized stars are still centered at leaves $9$ and $10$.

We also point out that this example satisfies an interesting property,
first observed by Colbourn \cite{col}.

\begin{prp}\label{degsort}
Let $G$ be a bipartite graph with bipartition $V=\{V_1,V_2\}$ and let $r\geq 1$. We say that $G$ has the \emph{bipartite degree sort} property if for all $x,y\in V_i$ with $d(x)\leq d(y)$, $\mathcal{J}^r_x(T)\geq \mathcal{J}^r_y(T)$.
\end{prp}

Not all bipartite graphs satisfy this property. Neiman \cite{nm} constructed the following counterexample, with $r=3$.

Fix positive integers $t$ and $k$ with $t\geq 2k\geq 4$. Let $G=G_{t,k}$ be the graph obtained from the complete bipartite graph $K_{2,t}$ and $P_{2k}$ by identifying one endpoint of $P_{2k}$ to be a vertex in $K_{2,t}$ lying in the bipartition of size $2$. Let $x$ be the other endpoint of the path, and let $y$ be a vertex in $K_{2,t}$ lying in the bipartition of size $t$, of degree $2$. An example is shown in Figure \ref{treex2}.
\begin{figure}[ht]
\centering
\includegraphics[height=5cm]{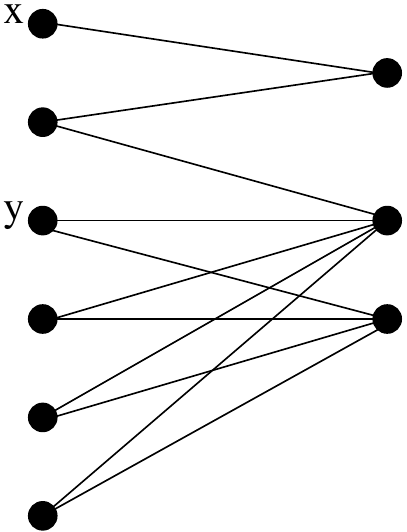}
\caption{$G_{4,2}$} \label{treex2}
\end{figure}

Let $Y=\mathcal{J}^3_y(G)$ and let $X=\mathcal{J}^3_x(G)$. We have, for $t\geq 2k$,
\begin{eqnarray}
Y-X &=& \mathcal{J}^2(G\downarrow y)-\mathcal{J}^2(G\downarrow x) \nonumber \\
&=& {t+2k-2 \choose 2}-|E(G\downarrow y)|-{t+2k-1 \choose 2}+|E(G\downarrow x)| \nonumber \\
&=& {t+2k-2 \choose 2}-{t+2k-1 \choose 2}+2t-1 \nonumber \\
&=& (t+2k-2)(-1)+2t-1 \nonumber \\
&=& t-2k+1 \nonumber \\
&>& 0.
\end{eqnarray}

We show that a similar construction acts as a counterexample for all $r>3$. Given $r>3$, consider the graph $G=G_{t,2}$, $t>r$. Let $x$ and $y$ be as defined before, with $d(x)=1$ and $d(y)=2$. Let $Y=\mathcal{J}^r_y(G)$ and $X=\mathcal{J}^r_x(G)$. We have $X={t+1 \choose r-1}$ and $Y={t+1 \choose r-1}+{t-1 \choose r-2}$. It follows that, for $t>r$, $Y>X$.

If we consider trees, it can be seen that the tree in Figure \ref{figtree} satisfies this property.
It is also not hard to show that the path $P_n$ satisfies this property, since for all $r\geq 1$, $\mathcal{J}^r_{v_1}(P_n)=\mathcal{J}^r_{v_n}(P_n)\geq \mathcal{J}^r_{v_i}(P_n)$ holds for each $2\leq i\leq n-1$.

Another infinite family of trees that satisfy the property are the depth-two stars shown in Figure \ref{treex} below.
\begin{figure}[ht]
\centering
\includegraphics[height=4cm]{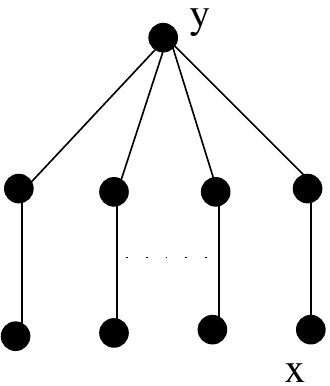}
\caption{Tree $T$ on $2n+1$ vertices which satisfies Conjecture \ref{degsort}.}\label{treex}
\end{figure}

Let $Y=\mathcal{J}^r_y(T)$ and let $X=\mathcal{J}^r_x(T)$. Then, we have $Y=\mathcal{J}^{r-1}(T\downarrow y)={n \choose r-1}$ and $X={n-1 \choose r-2}+2^{r-1}{n-1 \choose r-1}$. It is then easy to note that when $r\geq 1$, $X-Y\geq 0$.

However, it turns out that not all trees satisfy this property. A counterexample, for $n=10$ and $r=5$, is shown in Figure \ref{ct}.
\begin{figure}[ht]
\centering
\includegraphics[height=5cm]{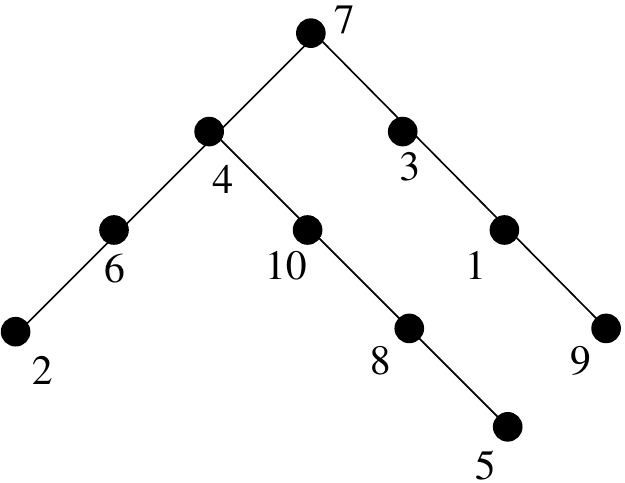}
\caption{Tree $T_1$ which does not satisfy Property \ref{degsort}} \label{ct}
\end{figure}
Observe that the vertex labeled $8$, with degree $2$, and the vertex labeled $4$, with degree $3$, lie in the same partite set, but we have $\mathcal{J}^5_4(T_1)=\{\{2,3,4,8,9\},\{2,3,4,5,9\}\}$ and $\mathcal{J}^5_8(T_1)=\{\{2,3,4,8,9\}\}$.
Note that, in this example, $r=\frac{n}{2}$. Another counterexample, with $n=12$ and $r=5$, is shown in Figure \ref{ct2}.
\begin{figure}[ht]
\centering
\includegraphics[height=4cm]{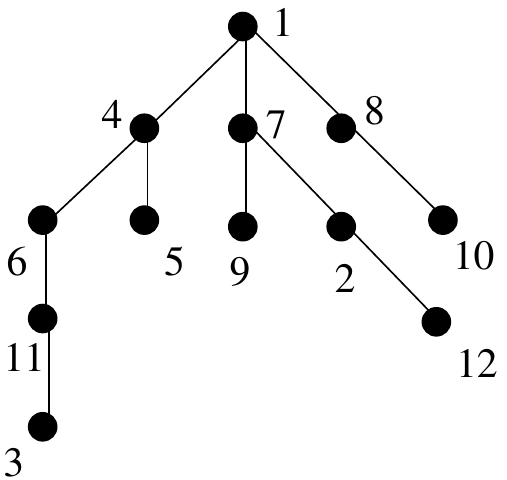}
\caption{Tree $T_2$ which does not satisfy Property \ref{degsort}} \label{ct2}
\end{figure}

We see that the vertices labeled $1$ and $2$, with degrees $3$ and $2$ respectively, lie in the same partite set. It can be checked that $|\mathcal{J}^5_1(T_2)|=32$ and $|\mathcal{J}^5_2(T_2)|=28$.
\subsection{Ladder graphs}
In this section, we give a proof of Theorem \ref{ladder},
which states that the ladder graph $L_n$ is $3$-EKR for all $n\geq 1$.
First, we state and prove a claim about maximum star families in $L_n$.

Let $G=L_n$ be a ladder with $n$ rungs. Let the rung edges be
$x_iy_i$($1\leq i \leq n$). First, we show that
$\mathscr{J}^r_{x}(G)$ is a maximum sized star for $x\in
\{x_1,y_1,x_n,y_n\}$.
\begin{claim}
If $G$ is a ladder with $n$ rungs,
$\mathscr{J}^r_{x}(G)$ is a maximum sized star for $x\in
\{x_1,y_1,x_n,y_n\}$.
\end{claim}
\begin{proof}
We prove the claim for $x=x_n$. The claim is obvious if $n\leq 2$,
so suppose $n\geq 3$. Let $\mathscr{A}$ be a star centered at some
$x\in V(G)$.
Without loss of generality,
we assume that $x=x_k$ for some $1<k<n$.
We now construct an injection from $\mathscr{A}$ to
$\mathscr{J}^r_{x_n}(G)$. Define functions $f$ and $g$ as follows.
\begin{displaymath}
   f(x) = \left \{
     \begin{array}{lr}
       x_{i \textrm{ mod n }+1} & \textrm{ if } x=x_{i} \\
       y_{i \textrm{ mod n }+1} & \textrm{ if } x=y_{i} \\
     \end{array}
   \right.
\end{displaymath}

\begin{displaymath}
   g(x) = \left \{
     \begin{array}{lr}
       y_{i} & \textrm{ if } x=x_{i} \\
       x_{i} & \textrm{ if } x=y_{i} \\
     \end{array}
   \right.
\end{displaymath}
Consider the function $f^{n-k}$. For every $A\in \mathscr{A}$,
define $f^{n-k}(A)=\{f^{n-k}(x):x\in A\}$ and similarly for $g$. We
define a function $h:\mathscr{A}\to \mathscr{J}^r_{x_n}(G)$ as
follows.
\begin{displaymath}
   h(A) = \left \{
     \begin{array}{lr}
       A          & \textrm{ if } \{x_1,x_n\} \subseteq A \\
       g(A)       & \textrm{ if } \{y_1,y_n\}\subseteq A \\
       f^{n-k}(A) & \textrm{ otherwise }
     \end{array}
   \right.
\end{displaymath}
Clearly, $x_n\in h(A)$ for every $A\in \mathscr{A}.$ We will
show that $h$ is an injection. Suppose $A,B\in \mathscr{A}$ and
$A\neq B$. We show that $h(A)\neq h(B)$. If both $A$ and $B$ are in
the same category(out of the three mentioned in the definition of
$h$), then it is obvious. So, suppose not. If $\{x_1,x_n\}\subseteq
A$ and $\{y_1,y_n\}\subseteq B$, then $x_k\in h(A)$, but $x_k\notin
h(B)$. Then, let $A$ be in either of the first two categories, and
let $B$ be in the third category. Then, $\{x_1,x_n\}\subseteq h(A)$,
but $\{x_1,x_n\}\not\subseteq h(B)$. This holds because otherwise,
we would have $\{x_k, x_{k+1}\}\subseteq B$, a contradiction.
\end{proof}

We give a proof of Theorem \ref{ladder}.

\begin{proof}
We do induction on the number of rungs. If $n=1$, we have $G=P_2$,
which is trivially $r$-EKR for $r=1$ and vacuously true for $r=2$
and $r=3$. Similarly, for $n=2$, $G=C_4$, so it is trivially $r$-EKR
for each $1\leq r\leq 2$ and vacuously true for $r=3$. So, let
$n\geq 3$. The case $r=1$ is trivial. If $r=2$, since $\delta(G)=2$
and $|G|\geq 6$, we can use Theorem \ref{2ekr} to conclude that $G$
is $2$-EKR. So consider $G$ such that $n\geq 3$ and $r=3$. If
$n=3$, the maximum size of an intersecting family of independent
sets of size $3$ is $1$, so $3$-EKR again holds trivially. So,
suppose $n\geq 4$. Let $G'=L_{n-1}$, $G''=L_{n-2}$. Also, let
$Z=\{x_{n-2},y_{n-2},x_{n-1},y_{n-1},x_n,y_n\}$. Define a function
$c$ as follows.
\begin{displaymath}
   c(x) = \left \{
     \begin{array}{lr}
       x_{n-1} & \textrm{ if } x=x_{n} \\
       y_{n-1} & \textrm{ if } x=y_{n} \\
       x     & \textrm{ otherwise }
     \end{array}
   \right.
\end{displaymath}
Let $\mathscr{A}\subseteq \mathscr{J}^r(G)$ be intersecting.

Define the following families.
$$\mathscr{B}=\{c(A):A\in \mathscr{A} \textrm{ and } c(A)\in
\mathscr{J}^r(G')\}$$
$$\mathscr{C}_1=\{A\setminus \{x_n\}:x_n\in A\in \mathscr{A} \textrm{ and }
A\setminus \{x_n\} \cup
\{x_{n-1}\}\in \mathscr{A}\}$$

$$\mathscr{C}_2=\{A\setminus \{y_n\}:y_n\in A\in \mathscr{A} \textrm{ and }
A\setminus \{y_n\} \cup
\{y_{n-1}\}\in \mathscr{A}\}$$
$$\mathscr{D}_1=\{A\in \mathscr{A}:A\cap Z=\{x_{n-2},x_n\} \}$$

$$\mathscr{D}_2=\{A\in \mathscr{A}:A\cap Z=\{y_{n-2},y_n\} \}$$
$$\mathscr{D}_3=\{A\in \mathscr{A}:A\cap Z=\{x_{n-1},y_n\} \}$$
$$\mathscr{D}_4=\{A\in \mathscr{A}:A\cap Z=\{y_{n-1},x_n\} \}$$
$$\mathscr{D}_5=\{\{x_{n-2}, y_{n-1},x_n\}\}$$
$$\mathscr{D}_6=\{\{y_{n-2}, x_{n-1},y_n\}\}$$

Define the families $\mathscr{E}=\mathscr{C}_1\cup
(\mathscr{D}_1-\{x_n\})$ and $\mathscr{F}=\mathscr{C}_2\cup
(\mathscr{D}_2-\{y_n\})$.
Then both
$\mathscr{E}\subseteq
\mathscr{J}^{r-1}(G'')$ and $\mathscr{F}\subseteq
\mathscr{J}^{r-1}(G'')$.
\begin{prop}\label{plad}
The family $\mathscr{E}$( $\mathscr{F}$) is a disjoint union of
$\mathscr{C}_1$ and $\mathscr{D}_1-\{x_n\}$($\mathscr{C}_2$ and
$\mathscr{D}_2-\{y_n\}$) and is intersecting.
\end{prop}
\begin{proof}
We prove the proposition for $\mathscr{E}$. The proof for
$\mathscr{F}$ follows similarly. Each $D\in \mathscr{D}_1-\{x_n\}$
contains $x_{n-2}$. However, no member in $\mathscr{C}_1$ contains
$x_{n-2}$. Thus, $\mathscr{E}$ is a disjoint union. To show that it
is intersecting, observe that $\mathscr{C}_1$ is intersecting since
for any $C_1,C_2\in \mathscr{C}_1$, $C_1\cup \{x_{n-1}\}$ and
$C_2\cup \{x_n\}$ are intersecting. Also, $\mathscr{D}_1-\{x_n\}$ is
intersecting since each member of the family contains $x_{n-2}$. So,
suppose $C\in \mathscr{C}_1$ and $D\in \mathscr{D}_1-\{x_n\}$. Then,
$C\cup \{x_{n-1}\}$ and $D\cup \{x_n\}$ are intersecting.
\end{proof}

\begin{prop}\label{p2lad}
If $G=L_n$, where $n\geq 4$, then we have
$$|\mathscr{J}^3_{x_1}(G)|\geq
|\mathscr{J}^3_{x_1}(G')|+2|\mathscr{J}^2_{x_1}(G'')|+2.$$
\end{prop}
\begin{proof}
Each $A\in \mathscr{J}^3_{x_1}(G')$ is also a member of
$\mathscr{J}^3_{x_1}(G)$, containing neither $x_n$ nor $y_n$. Each
$A\in \mathscr{J}^3_{x_1}(G'')$ contributes two members to
$\mathscr{J}^3_{x_1}(G)$, $A\cup \{x_n\}$ and $A\cup \{y_n\}$. Also,
$\{x_1,x_{n-1},y_n\},\{x_1, y_{n-1}, x_n\}\in
\mathscr{J}^3_{x_1}(G)$. This completes the argument.
\end{proof}

We have
\begin{eqnarray}
|\mathscr{A}|&=&|\mathscr{B}|+\sum_{i=1}^2|\mathscr{C}_i|+\sum_{i=1}^6|\mathscr{
D}_i|
\nonumber \\
&=&|\mathscr{B}|+|\mathscr{E}|+|\mathscr{F}|+\sum_{i=3}^6|\mathscr{D}_i|.
\label{eqlad}
\end{eqnarray}

We consider two cases.
\begin{itemize}
\item $\mathscr{D}_3\neq \emptyset$ and $\mathscr{D}_4\neq
\emptyset$.

In this case, we must have $\mathscr{D}_3=\{\{a,x_{n-1},y_n\}\}$ and
$\mathscr{D}_4=\{\{a,y_{n-1},x_n\}\}$ for some $a\notin
\{y_{n-2},x_{n-2}\}$ and hence, $|\mathscr{D}_3|=|\mathscr{D}_4|=1$.
Also
$\mathscr{D}_5=\mathscr{D}_6=\emptyset$.
So, using Equation  \ref{eqlad}, Propositions \ref{plad} and \ref{p2lad} and
the induction hypothesis, we have
\begin{eqnarray*}
|\mathscr{A}|&=&|\mathscr{B}|+|\mathscr{E}|+|\mathscr{F}|+\sum_{i=3}^6|\mathscr{
D}_i| \\
&\leq& |\mathscr{J}^r_{x_1}(G')|+2|\mathscr{J}^{r-1}_{x_1}(G'')|+2 \\
&\leq& |\mathscr{J}^r_{x_1}(G)|.
\end{eqnarray*}

\item Without loss of generality,
we suppose that $\mathscr{D}_4=\emptyset.$ If
$\mathscr{D}_3=\emptyset$, then $\sum_{i=3}^6|\mathscr{D}_i|\leq 1$,
so we are done by Proposition \ref{p2lad}.
So, suppose $|\mathscr{D}_4|>0$.
We again consider two cases.
\begin{enumerate}
\item Suppose $\mathscr{C}_1=\emptyset$ and
$\mathscr{D}_1=\emptyset$.

We note that at most one out of $\mathscr{D}_5$ and $\mathscr{D}_6$
can be nonempty. We also note that $|\mathscr{D}_3|\leq 2(n-3)$ and
$\mathscr{J}^2_{x_1}(G'')=2(n-3)-1$. So, using Proposition
\ref{p2lad}
\begin{eqnarray*}
|\mathscr{A}|&=&|\mathscr{B}|+|\mathscr{F}|+|\mathscr{D}_3|+1 \\
&\leq&
|\mathscr{J}^r_{x_1}(G')|+|\mathscr{J}^{r-1}_{x_1}(G'')|+2(n-3)+1 \\
&\leq & |\mathscr{J}^r_{x_1}(G)|.
\end{eqnarray*}

\item Suppose that either $\mathscr{C}_1\neq \emptyset$ or
$\mathscr{D}_1\neq \emptyset.$ Let $C=\{a,b\}\in \mathscr{C}_1$ and
$D\in \mathscr{D}_3$. We have $C\cup \{x_n\} \cap D\neq \emptyset$.
So, we have $D\setminus \{y_n,x_{n-1}\}=\{a\}$ or $D\setminus
\{y_n,x_{n-1}\}=\{b\}.$ So, $|\mathscr{D}_3|\leq 2$. If
$|\mathscr{D}_3|=2$, then $y_{n-2}\notin \{a,b\}$, so
$\mathscr{D}_6=\emptyset.$ Also, $\mathscr{D}_5=\emptyset$ since
$\mathscr{D}_3$ is nonempty. If $|\mathscr{D}_3|\leq 1$, then
$|\mathscr{D}_6|\leq 1$. Thus, in either case,
$\sum_{i=3}^6|\mathscr{D}_i|\leq 2$. Thus, using Equation
\ref{eqlad} and Proposition \ref{p2lad}, we are done. A similar
argument works if $\mathscr{D}_1$ is nonempty.
\end{enumerate}
\end{itemize}
\end{proof}


\begin{thebibliography}{99}

\bibitem{berge} C. Berge, Nombres de coloration de l'hypergraphe
$h$-partie complet, Hypergraph Seminar, Columbus, Ohio 1972,
Springer, New York, 1974, pp. 13-20.

\bibitem{bh} P.Borg, F.Holroyd, The Erd\"{o}s-Ko-Rado properties of various graphs containing singletons,
Discrete Mathematics (2008), doi:10.1016/j.disc.2008.07.021

\bibitem{col} C. Colbourn, personal communication.

\bibitem{df} M. Deza, P. Frankl, Erd\"{o}s-Ko-Rado theorem -- 22
years later, SIAM J. Algebraic Discrete Methods 4 (1983), no. 4,
419-431.

\bibitem{dirac} G. A. Dirac, On rigid circuit graphs. Abh. Math. Sem. Univ. Hamburg 25(1961), 71-76.

\bibitem{ekr} P. Erd\"{o}s, C. Ko, R. Rado, Intersection theorems
for systems of finite sets, Quart. J. Math Oxford Ser. (2) 12(1961),
313-320.

\bibitem{gron} H. D. O. F. Gronau, More on the Erd\"{o}s-Ko-Rado
theorem for integer sequences, J. Combin. Theory Ser. A 35 (1983),
279-288.

\bibitem{hm} A. J. W. Hilton, E. C. Milner, Some intersection
theorems for systems of finite sets, Quart. J. Math. Oxford 18
(1967), 369-384.

\bibitem{hs} A. J. W. Hilton, C. L. Spencer, A graph-theoretical generalization
of Berge's analogue of the Erd\"{o}s-Ko-Rado theorem.  Graph theory
in Paris, Trends Math., Birkh\"{a}user, Basel (2007), 225--242.

\bibitem{hst} F.C. Holroyd, C. Spencer, J. Talbot, Compression and
Erd\"{o}s-Ko-Rado Graphs,  Discrete Math. 293 (2005), no. 1-3,
155-164

\bibitem{ht} F.C. Holroyd, J. Talbot,
Graphs with the Erd\"{o}s-Ko-Rado property,  Discrete Math. 293
(2005), no. 1-3, 165-176.

\bibitem{living} M. L. Livingston, An ordered version of the
Erd\"{o}s-Ko-Rado theorem, J. Combin. Theory Ser. B 26 (1979)
162-165.


\bibitem{meyer} J. C. Meyer, Quelques probl\`{e}mes concernant les
cliques des hypergraphes $k$-complets et $q$-parti $h$-complets,
Hypergraph Seminar, Columbus, Ohio, 1972, Springer, New York, 1974,
pp. 127-139.

\bibitem{moon} A. Moon, An analogue of the Erd\"{o}s-Ko-Rado theorem
for the Hamming schemes $H(n,q)$, J. Combin. Theory Ser. A 32(1982)
386-390.

\bibitem{nm} M. Neiman, personal communication.

\bibitem{tal} J. Talbot, Intersecting families of separated sets, J.
London Math. Soc. (2) 68 (2003), no. 1, 37-51.

\end{thebibliography}
\end{document}